\documentclass[11pt,a4paper,reqno]{amsart}
\usepackage[T1]{fontenc}
\usepackage{lmodern}
\usepackage{geometry}
\usepackage[utf8]{inputenc}
\usepackage{CJK} 

\usepackage{microtype}

\usepackage{amsmath}

\usepackage{amsfonts}

\usepackage{amsmath,amsthm,amssymb}
\usepackage[abbrev]{amsrefs}

\allowdisplaybreaks

\numberwithin{equation}{section}

\newtheorem{theorem}{Theorem}[section]
\newtheorem{lemma}[theorem]{Lemma}
\newtheorem{proposition}[theorem]{Proposition}
\theoremstyle{definition}

\newtheorem{remark}[theorem]{Remark}

\numberwithin{equation}{section}

\newcommand{\abs}[1]{\lvert#1\rvert}
\newcommand{\bigabs}[1]{\bigl\lvert#1\bigr\rvert}
\newcommand{\st}{\;\vert\;}
\newcommand{\dif}{\,\mathrm{d}}
\DeclareMathOperator{\supp}{supp}

\title[Choquard equations under confining potentials]
      {Choquard equations under confining external potentials}

\author[J. Van Schaftingen]{Jean Van Schaftingen}
\address{Jean Van Schaftingen\\ Institut de Recherche en Math\'{e}matique et Physique\\
Universit\'{e} Catholique de Louvain, Chemin du Cyclotron 2 bte L7.01.01, 1348 Louvain-la-Neuve, Belgium}
\email{Jean.VanSchaftingen@UCLouvain.be}

\author[J. Xia]{Jiankang Xia (夏健康)}
\address{Jiankang Xia\\ Chern Institute of Mathematics\\
Nankai University, Tianjin, 300071, China}
\email{fyxt001@163.com}

\begin{document}

\begin{CJK}{UTF8}{gbsn}

\begin{abstract}
We consider the nonlinear Choquard equation
$$
-{\Delta}u+V u=\bigl(I_\alpha \ast \vert u\vert ^p\bigr)\vert u\vert ^{p-2}u \qquad \text{ in } \mathbb{R}^N
$$
where $N\geq 1$, $I_\alpha$ is the Riesz potential integral operator of order $\alpha \in (0, N)$ and $p > 1$.
If the potential $ V \in C (\mathbb{R}^N; [0,+\infty)) $ satisfies the confining condition
$$
\liminf\limits_{\vert x\vert \to +\infty}\frac{V(x)}{1+\vert x\vert ^{\frac{N+\alpha}{p}-N}}=+\infty,
$$
and \(\frac{1}{p} > \frac{N - 2}{N + \alpha}\),
we show the existence of a groundstate, of an infinite sequence of solutions of unbounded energy and, when \(p \ge 2\) the existence of least energy nodal solution.
The constructions are based on suitable weighted compact embedding theorems.
The growth assumption is sharp in view of a Poho\v{z}aev identity that we establish.
\end{abstract}

\maketitle

\end{CJK}
 
\section{Introduction and main results}
We are interested in the following class of Choquard equations
\begin{align}
\tag{$\mathcal{C}$}\label{eq1.1}
  -{\Delta} u + V u&=\bigl(I_\alpha \ast \vert u\vert ^p\bigr)\vert u\vert ^{p-2}u&
  &\text{ in } \mathbb{R}^N
\end{align}
in the Euclidean space \(\mathbb{R}^N\) of dimension $N\geq 1$, where
$I_\alpha:\mathbb{R}^N\to \mathbb{R}$ is the Riesz potential of order $\alpha\in(0,N)$, which is defined for every $x\in\mathbb{R}^N\backslash{\{0\}}$ by
\[
  I_{\alpha}(x)=\frac{A_\alpha}{\vert x\vert ^{N-\alpha}}, \text{ with } A_\alpha=\frac{\Gamma(\frac{N-\alpha}{2})}{\Gamma(\frac{\alpha}{2})\pi^{\frac{N}{2}}}
\]
where $\Gamma$ denotes the classical Gamma function, and $p>1$ is a given exponent.

When \(N = 3\), \(\alpha = 2\) and \(p = 2\), the equation \eqref{eq1.1} appears
in several physical contexts, such as standing waves for the Hartree equation, the description by Pekar of the quantum physics of a polaron
at rest \cite{P}, the description by Choquard of an electron trapped in its own hole \cite{L} or the coupling of the Schr\"odinger equation
under a classical Newtonian gravitational potential \cites{Diosi1984,J1,J2,MPT,Penrose1996}.

When the potential \(V\) is a positive constant function, groundstate solutions are known to exist \cites{L,Lions,MV}
under the assumption that the exponent \(p\) satisfies
\begin{equation}
\label{eqInterCritical}
   \frac{N - 2}{N + \alpha} < \frac{1}{p} < \frac{N}{N + \alpha}.
\end{equation}
Moreover, infinitely many geometrically distinct solutions can be constructed \cites{Lions}.
We refer the reader to the survey \cite{MVSReview} for further discussion and references on the Choquard equation.

The goal of the present work is to examine how the presence of a confining potential \(V\) changes and possibly improves the situation.
Our first result is that groundstates can exist in a \emph{wider range of nonlinearities} when the external potential \(V\) is coercive enough.

\begin{theorem}\label{thm1.2}
Let $N\geq 1$, \(\alpha \in (0, N)\), $p \in (1, +\infty)$ and \(V \in C (\mathbb{R}^N;[0, +\infty))\). If
\[
  \frac{1}{p} > \frac{N - 2}{N + \alpha}
\]
and if
\[
  \liminf\limits_{\vert x\vert \to +\infty}\frac{V(x)}{1+\vert x\vert ^{\frac{N+\alpha}{p}-N}}=+\infty,
\]
then the Choquard equation \eqref{eq1.1} has a groundstate solution. %
\end{theorem}

The solutions are groundstates in the sense that they minimize among nontrivial solutions the functional
\[
  J_p(u)=\frac{1}{2}\int_{\mathbb{R}^N}\vert \nabla u\vert ^2+V\vert u \vert^2 -\frac{1}{2p}\int_{\mathbb{R}^N}\bigl(I_\alpha \ast \vert u\vert ^p\bigr)\vert u\vert ^p;
\]
solutions of the Choquard equation \eqref{eq1.1} are formally critical points of the functional
\(J_p\).

A striking feature of Theorem~\ref{thm1.2} is that the condition \(p > \frac{N + \alpha}{N}\) in \eqref{eqInterCritical} can be loosened when \(V\) grows fast enough at infinity. In particular, if \(V (x) = \abs{x}^\beta\) with \(\beta > 0\), one can take \(p > \max\{\frac{N + \alpha}{N + \beta}, 1\}\).
The growth assumption is sharp. Indeed, if \(V (x) = \abs{x}^\beta\)
and  $u\in W^{2,2}_{\mathrm{loc}}(\mathbb{R}^N)\cap H^{1}_{V}(\mathbb{R}^N)$ solves the Choquard equation \eqref{eq1.1} and
then we have the Poho\v{z}aev identity (Theorem~\ref{thm5.1})
\begin{equation*}
\frac{N-2}{2}\int_{\mathbb{R}^N} \vert \nabla u\vert ^2 +\frac{N + \beta}{2}\int_{\mathbb{R}^N} V \vert u \vert^2
=\frac{N+\alpha}{2p}\int_{\mathbb{R}^N} \bigl(I_\alpha \ast \vert u\vert ^p\bigr)\vert u\vert ^p ,
\end{equation*}
provided that the integral on the right-hand side is finite.
This condition cannot be satisfied when \(p\) does not satisfy the assumption of Theorem~\ref{thm1.2}.

Theorem~\ref{thm1.2} can be thought as counterpart for the nonlocal Choquard equation of results for the nonlinear Schr\"odinger equation with a coercive potential \cites{R}.
Radial positive solutions for the Choquard equation \eqref{eq1.1} had already been obtained in the quadratic case \(p = 2\) when the potential \(V\) is radial and radially increasing \cite{CaoWangZou}.

The core of the proof of Theorem~\ref{thm1.2} is to obtain the well-definiteness, the continuity and the compactness properties of the Riesz potential energy term in the definition of the functional \(J_p\).
This is done by combining a suitable Sobolev-type compact weighted embedding theorem together with the weighted estimates for fractional integrals of Stein and Weiss \cites{SW}, which are a weighted counterpart of the more classical Hardy--Littlewood--Sobolev inequality.

\medskip

We now turn on to the question whether the Choquard equation has, under the conditions of Theorem~\ref{thm1.2} more solutions. This is indeed the case and there are infinitely many solutions.

\begin{theorem}\label{thm1.3}
Let $N\geq 1$, \(\alpha \in (0, N)\), $p \in (1, +\infty)$ and \(V \in C (\mathbb{R}^N; [0, +\infty))\). If
\[
  \frac{1}{p} > \frac{N - 2}{N + \alpha}
\]
and if
\[
  \liminf\limits_{\vert x\vert \to+\infty}\frac{V(x)}{1+\vert x\vert ^{\frac{N+\alpha}{p}-N}}=+\infty,
\]
then the Choquard equation \eqref{eq1.1} has an infinite sequence of solutions
whose energies do not remain bounded.
\end{theorem}

The solutions are constructed with the fountain theorem \cite{B} (see also \cite{W}*{Theorem 3.6}); thanks to the same weighted embedding and fractional integral estimates as in the proof of Theorem~\ref{thm1.2}, the Palais--Smale condition for the functional $J_p$ can be established by classical arguments.

\medskip

Finally we investigate the question whether the Choquard equation \eqref{eq1.1} has a least energy sign-changing solution, that is, a solution that changes sign and which minimizes the functional \(J_p\) among such solutions.

A natural way to construct such solutions is to minimize, as for the local semilinear elliptic problems \cites{CSS,CCN,LW,SW2}, the functional on the \emph{Nehari nodal set}:
\[
 \bigl\{ u \in H^1 (\mathbb{R}^N) \st u^+ \ne 0, \;u^- \ne 0,\; \langle J_p'(u),u^{+}\rangle= 0\text{ and }\langle J_p'(u),u^{-}\rangle= 0 \bigr\},
\]
where \(u^+ = \max (u, 0)\) and \(u^- = \min (u, 0)\).
Such solutions of \eqref{eq1.1} have been constructed when \(V = 1\)
\[
\frac{N - 2}{N + \alpha} <  \frac{1}{p} \le \frac{1}{2};
\]
they were obtained by a new minimax principle and concentration-compactness method, and the minimization problem on the Nehari nodal set was observed to be degenerate when \(p < 2\) \cites{GMV,GV}.

Sign-changing solutions have been constructed for the Schr\"{o}dinger-Poisson system in $\mathbb{R}^3$ in which a nonlocal nonlinearity appears with opposite sign \cites{WZ,AS}.

Now, we are in a position to state our main results on the nodal solutions of equation \eqref{eq1.1}:

\begin{theorem}\label{thm1.5}
Let $N\geq 1$, \(\alpha \in ((N - 4)_+, N)\), $p \in [2, +\infty)$ and \(V \in C (\mathbb{R}^N; [0, +\infty))\). If
\[
  \frac{1}{p} > \frac{N - 2}{N + \alpha}
\]
and if
\[
  \liminf\limits_{\vert x\vert \to +\infty} V(x)=+\infty,
\]
then the Choquard equation \eqref{eq1.1} has at least one least-energy sign-changing solution.%
\end{theorem}

As before, the assumptions provide us with a functional with nice compactness properties. The situation is still more challenging than for a local semilinear elliptic equation on a bounded domain because some of the usual properties of the local nonlinear Schr\"odinger functional on negative and positive parts fail: in general \(J_p(u)\ne J_p (u^+)+J_p(u^-)\), and \(\langle J_p'(u),u^{\pm}\rangle\ne\langle J_p'(u^{\pm}),u^{\pm}\rangle\).

Theorem~\ref{thm1.5} was stated by Ye \cite{Y}*{Theorem 1.3}; it seems that his argument unfortunately overlooks the crucial question whether the proposed solution \(u\) does change sign, which is quite delicate when \(p = 2\) (see the proof of Theorem~\ref{thm1.5} and \cite{GMV}). We propose here a proof relying on tools similar to those for Theorems~\ref{thm1.2} and \ref{thm1.3}.

When \(p < 2\), we prove that the energy functional does not achieve its minimum on the Nehari nodal set (see Proposition~\ref{prop4}).

\medskip

The remainder of this paper is organized as follows.  In section \ref{sectionEmbedding},
we first prove a weighted embedding theorem, then show that the function $J_p$ is of $C^1$ on the natural Sobolev space $H_{V}^{1}(\mathbb{R}^N)$ and satisfies the Palais--Smale condition. The proof of our main results will be postponed to the next two sections \ref{sectionGroundstates} and \ref{sectionNodal}. In the last section \ref{sectionPohozaev}, we will establish Poho\v{z}aev identity responding to equation \eqref{eq1.1}, with which we can deduce some nonexistence results.

\section{Function spaces and weighted embedding theorems}
\label{sectionEmbedding}
The linear part of the Choquard equation \eqref{eq1.1} naturally induces the Euclidean norm
\[
\Vert u\Vert _{V}:=\Big(\int_{\mathbb{R}^N}\vert \nabla u\vert ^2+V\vert u \vert^2 \Big)^{\frac{1}{2}}.
\]
We define $H^{1}_{V}(\mathbb{R}^N)$ as the Hilbert space obtained by
completion of the set of smooth test functions $C^{\infty}_{c}(\mathbb{R}^N)$ with respect to the norm $\Vert \cdot\Vert _V$. We first establish some embedding theorem from $H^{1}_{V}(\mathbb{R}^N)$ into the weighted spaces $L^2(\vert x\vert ^\gamma\dif x;\mathbb{R}^N)$ which is defined for $\gamma\geq0$ by
\begin{equation*}
L^2 \bigl(\vert x\vert ^\gamma\dif x;\mathbb{R}^N\bigr):=\Bigl\{u : \mathbb{R}^N \to \mathbb{R}\st u \text{ is measurable  and } \int_{\mathbb{R}^N}\vert x\vert ^\gamma \vert u(x)\vert ^2 \dif x<+\infty\Bigr\}.
\end{equation*}

We begin by establishing the following embedding theorem.

\begin{proposition}\label{embedding}
Let \(N \ge 1\) and \(\gamma \in [0, +\infty)\). If \(V \in C (\mathbb{R}^N; [0, +\infty))\) satisfies
\[
  \liminf\limits_{\vert x\vert \to+\infty} \frac{V(x)}{\vert x \vert^\gamma}>0,
\]
then there exists a constant \(C > 0\) such that for every \(u \in H^1_V (\mathbb{R}^N)\),
\[
  \int_{\mathbb{R}^N} \vert x\vert^\gamma \vert u (x) \vert^2 \dif x
  \le C \int_{\mathbb{R}^N} \vert \nabla u\vert^2 + V \vert u \vert^2.
\]
If moreover,
\[
  \lim_{\vert x\vert \to+\infty} \frac{V(x)}{\vert x \vert^\gamma} = +\infty,
\]
then the corresponding embedding is compact. In particular, the embedding $H_{V}^{1}(\mathbb{R}^N)\hookrightarrow L^q(\mathbb{R}^N)$ is compact for any
$q$ with $\frac{1}{q}\in(\frac{1}{2}-\frac{1}{N},\frac{1}{2})$ if $\gamma=0$.
\end{proposition}
\begin{proof}
Given \(\lambda \in (0, +\infty)\) such that
\[
  \lambda < \liminf\limits_{\vert x\vert \to+\infty}\frac{V(x)}{\vert x \vert^\gamma},
\]
there exists \(\kappa > 0\) sufficiently large so that if \(x \in \mathbb{R}^N \setminus B (0,  \frac{\kappa}{2} )\), we have
\(V(x)\geq \lambda \vert x \vert^\gamma \).
(Here and in the sequel, we use the notation  $B(a,r)$ for the ball centered at \(a\) of radius $r$  and in $\mathbb{R}^N$.)
By integration, we have in particular,
\begin{equation}
\label{ineqLambdaOutside}
\lambda \int_{\mathbb{R}^N \setminus B (0,\frac{\kappa}{2} )}\vert x \vert^\gamma \vert u (x) \vert^2 \dif x\leq \int_{\mathbb{R}^N \setminus B (0, \frac{\kappa}{2} )}V\vert u \vert^2.
\end{equation}
We take a function $\varphi\in C^{\infty}(\mathbb{R}^N)$
such that $0\leq \varphi\leq1$ in \(\mathbb{R}^N\), $\varphi(x)=1$ for every $x\in B(0,\frac{\kappa}{2})$ and $\varphi (x)= 0$ for every $x \in \mathbb{R}^N \setminus B (0,  \kappa)$.
Then, it follows that
\begin{equation*}
\begin{split}
\int_{\mathbb{R}^N} &\vert x \vert^\gamma \vert u (x) \vert^2\dif x \le \kappa^\gamma \int_{B (0, \kappa)} \varphi^2\vert u \vert^2+\int_{\mathbb{R}^N \setminus B (0, \frac{\kappa}{2})} \vert x \vert^\gamma \vert u (x) \vert^2 \dif x\\
&\leq C_1\kappa^\gamma   \int_{B (0, \kappa)} \vert \nabla(\varphi u)\vert ^2+ \int_{\mathbb{R}^N \setminus B (0, \frac{\kappa}{2} )} \vert x \vert^\gamma \vert u (x) \vert^2 \dif x\\
&\leq 2 C_1\kappa^\gamma \int_{B (0,  \kappa)}\vert \nabla u\vert ^2+2 C_1\kappa^\gamma \int_{B (0, \kappa) \setminus B (0, \frac{\kappa}{2})}\vert \nabla \varphi\vert ^2\vert u \vert^2
+ \int_{\mathbb{R}^N \setminus B (0, \frac{\kappa}{2} )} \vert x \vert^\gamma \vert u (x) \vert^2\dif x\\
&\leq 2 C_1\kappa^\gamma \int_{\mathbb{R}^N}\vert \nabla u\vert ^2+ \biggl(2 C_1\kappa^\gamma \frac{\Vert \nabla \varphi\Vert_{L^\infty (\mathbb{R}^N)}^2}{(\kappa/2)^\gamma} + 1\biggr)
 \int_{\mathbb{R}^N \setminus B (0, \frac{\kappa}{2} )} \vert x \vert^{\gamma} \vert u (x) \vert^2\dif x.
\end{split}
\end{equation*}
where the constant $C_1$ comes from the Poincar\'{e} inequality with Dirichlet boundary conditions on the ball \(B (0,   \kappa)\),
 which is independent of the function $u$.
We now apply the estimate \eqref{ineqLambdaOutside} to the second term to obtain
\begin{equation}
\begin{split}
\label{eq2.3}
\int_{\mathbb{R}^N} \vert x \vert^\gamma \vert u (x) \vert^2\dif x & \leq  2 C_1 \kappa^\gamma \int_{\mathbb{R}^N}\vert \nabla u\vert ^2+
 \frac{2^{\gamma + 1} C_1\Vert \nabla \varphi\Vert_{L^\infty (\mathbb{R}^N)}^2+ 1}{\lambda} \int_{\mathbb{R}^N}V\vert u \vert^2\\
& \le \max \,\biggl\{2 C_1 \kappa^\gamma , \frac{2^{\gamma + 1} C_1\Vert \nabla \varphi\Vert_{L^\infty (\mathbb{R}^N)}^2+ 1}{\lambda}\biggr\}
\int_{\mathbb{R}^N} \vert \nabla u\vert^2 + V \vert u \vert^2,
\end{split}
\end{equation}
and the first part of the conclusion follows.

For the compactness, without loss of generality, let $(v_n)_{n \in \mathbb{N}}\) be a sequence such that $v_n\rightharpoonup 0$ weakly as \(n \to \infty\) in $H^{1}_{V}(\mathbb{R}^N)$. In particular, the sequence $(v_n)_{n \in \mathbb{N}}$ is bounded in \(H^1_V (\mathbb{R}^N)\).
We are going to prove that $v_n\to 0$ strongly as \(n \to \infty\) in $L^{2}(\vert x\vert ^{\gamma}\dif x; \mathbb{R}^N)$.
By assumption, for every $\varepsilon>0$, there exists $R_1>0$, such that
\[
\bigg(\sup_{\vert x\vert \geq R_1}\frac{\vert x\vert ^\gamma}{V(x)}\bigg)\Vert v_n\Vert _{V}^2\leq \varepsilon.
\]
Since $\gamma\geq0$, for any fixed $R>0$, the weighted space $L^{2}(\vert x\vert ^{\gamma}\dif x;B(0,R))$ is embedded into the classical Lebesgue space $L^2(B(0,R))$ defined on bounded domain $B(0,R)$. By the classical Sobolev embedding theorem, $v_n\to 0$ strongly in $L^2(B(0,R))$ as \(n \to \infty\). Therefore, for fixed $R\geq R_1$, there exists $N_1>0$ such that
\[
\int_{B(0,R)} \vert x\vert ^\gamma \vert v_{n}(x)\vert ^2\dif x\leq \varepsilon \qquad \text{ for each } n\geq N_1.
\]
Then for $n\geq N_1$, we have
\begin{equation*}
\begin{split}
 \int_{\mathbb{R}^N} \vert x\vert ^\gamma \vert v_{n}(x)\vert ^2\dif x&=\int_{B(0,R)} \vert x\vert ^\gamma \vert v_{n}(x)\vert ^2\dif x+\int_{\mathbb{R}^N\backslash{B(0,R)}} \vert x\vert ^\gamma \vert v_n (x)\vert^2 \dif x \nonumber\\
  &\leq \int_{B(0,R)} \vert x\vert ^\gamma \vert v_{n}(x)\vert ^2 \dif x+\bigg(\sup_{\vert x\vert \geq R}\frac{\vert x\vert ^\gamma}{V(x)}\bigg)\int_{\mathbb{R}^N\backslash{B(0,R)}} V(x) \vert v_n (x)\vert^2 \dif x \nonumber\\
  &\leq \varepsilon+\bigg(\sup_{\vert x\vert \geq R}\frac{\vert x\vert ^\gamma}{V(x)}\bigg)\Vert v_n\Vert _{V}^2\leq 2\varepsilon.\nonumber
\end{split}
\end{equation*}

Finally, we interpolate to conclude our proof of the compact embedding $H^{1}_{V} \subset L^q(\mathbb{R}^N)$ with
$\frac{1}{2}-\frac{1}{N}<\frac{1}{q}<\frac{1}{2}$ for the case $\gamma=0$.  Take $\bar{q}=\frac{2N}{N-2}$ if $N\geq 3$, or any $\bar{q}\in(q, +\infty)$ if $N=1,2$, there exists $b\in(0,1)$ such that
\[
\frac{1}{q}=\frac{b}{2}+\frac{1-b}{\bar{q}},
\]
it follows that as $n\to\infty$,
\begin{equation*}
\Vert v_n\Vert _{L^{q}}
\leq \Vert v_n\Vert ^{b}_{L^2}\Vert v_n\Vert ^{1-b}_{L^{\bar{q}}}
\leq C^{(1-b)}  \Vert v_n\Vert ^{b}_{L^2}\Vert v_n\Vert _{V}^{1-b}\to 0.\qedhere
\end{equation*}
\end{proof}

With the aid of the Stein--Weiss inequality \cites{SW}, we show that the nonlocal Riesz potential energy term $\mathcal{G}_p$ of the functional $J_p$ is well-defined
and prove that the functional $J_p$ is of class $C^1$ on the weighted Sobolev space $H_{V}^{1}(\mathbb{R}^N)$. Finally, thanks to the compact embedding result, we close this section by
verifying that the functional $J_p$ satisfies the Palais--Smale condition.

\begin{proposition}\label{lem2.2}
Let $N\geq 1$ and \(\alpha \in (0, N)\). If $V \in C (\mathbb{R}^N; [0, +\infty))$ satisfies
\[
\liminf\limits_{\vert x\vert \to +\infty}\frac{V(x)}{1+\vert x\vert ^{\frac{N+\alpha}{p}-N}}>0,
\]
then the mappings $u\in H^{1}_{V}(\mathbb{R}^N)\longmapsto I_{\alpha/2}*\vert u\vert ^p\in L^2(\mathbb{R}^N)$ and
\[
u\in H^{1}_{V}(\mathbb{R}^N)\longmapsto \bigl(I_\alpha \ast \vert u\vert ^p\bigr)\vert u\vert ^{p-2}u\in (H_{V}^{1}(\mathbb{R}^N))'
\]
are continuous for $p>1$ and $\frac{1}{p}>\frac{N-2}{N+\alpha}$.\\
If moreover
\[
\liminf\limits_{\vert x\vert \to +\infty}\frac{V(x)}{1+\vert x\vert ^{\frac{N+\alpha}{p}-N}}=+\infty,
\]
the above mappings are weak to strong type, that is, they map weakly converging sequence to strongly converging sequence.
\end{proposition}

Here and in the sequel, $X'$ denotes the topological dual space of the normed space $X$.

\begin{proof}[Proof of Proposition~\ref{lem2.2}]
In the case \(p > \frac{N + \alpha}{N}\), the well-definiteness and the continuity follow from the continuous embedding \(H^1_V (\mathbb{R}^N) \subset H^1 (\mathbb{R}^N)\), the classical Sobolev embedding and the Hardy--Littlewood--Sobolev inequality as in the case where \(V\) is constant \cite{MorozVanSchaftingen2013JFA}.
If moreover \(\liminf_{\vert x \vert \to +\infty} V (x) = +\infty\), the embedding \(H^1_V (\mathbb{R}^N) \subset L^q (\mathbb{R}^N)\) is compact
for every \(q \in [2, +\infty)\) with \(\frac{1}{q} > \frac{1}{2} - \frac{1}{N}\) by Proposition~\ref{embedding}, and then the weak to strong continuity property follows.

We assume now that \(p \leq \frac{N+\alpha}{N}<2\).
We first show that the nonlocal term $\mathcal{G}_p$ of the functional $J_p$ is well defined on the space $H^{1}_{V}(\mathbb{R}^N)$.
By the Stein--Weiss inequality \cite{SW}, together with the semi-group identity for the Riesz potential
$I_\alpha=I_{\alpha/2}*I_{\alpha/2}$ \cites{LL}, we have, since \(\frac{2}{p} > 1\),
\begin{equation*}
\mathcal{G}_p(u):=\int_{\mathbb{R}^N}\bigl(I_\alpha \ast \vert u\vert ^p\bigr)\vert u\vert ^p
=\int_{\mathbb{R}^N} \vert I_\frac{\alpha}{2}*\vert u\vert ^p\vert ^2\leq C\Big( \int_{\mathbb{R}^N} \vert x\vert ^{\frac{N + \alpha}{p} - N}\vert u(x)\vert ^2\dif x \Big)^{p}.
\end{equation*}
In view of the above continuous embedding Proposition~\ref{embedding}, the functional $\mathcal{G}_p$ is well defined on $H^{1}_{V}(\mathbb{R}^N)$.
By Proposition~\ref{embedding} again, the superposition operator
\begin{equation}
\label{eqSuperposition}
u \in L^{2}\bigl(\vert x\vert ^{\frac{N + \alpha}{p} - N}\dif x;\mathbb{R}^N\bigr)\longmapsto \vert u \vert^p \in L^{\frac{2}{p}}\bigl(\vert x\vert ^{\frac{N + \alpha}{p} - N}\dif x;\mathbb{R}^N\bigr)
\end{equation}
is continuous. Taking into account the Stein--Weiss inequality \cites{SW} again, the Riesz potential integral operator
\begin{equation}
\label{eqRieszMap}
f\in L^{\frac{2}{p}}\bigl(\vert x\vert ^{\frac{N + \alpha}{p} - N}\dif x; \mathbb{R}^N\bigr)\longmapsto I_{\alpha/2}*f \in L^2(\mathbb{R}^N)
\end{equation}
is a continuous linear operator. Thus the conclusion follows and the stronger conclusion follows directly from the compact embedding Proposition~\ref{embedding}.
In fact, suppose that $u_n\rightharpoonup u$ weakly in $H^{1}_{V}(\mathbb{R}^N)$, by compactness, we know that, as \(n \to \infty\)
\[
u_n\to u \text{ strongly  in } L^{2}\bigl(\vert x\vert ^{\frac{N + \alpha}{p} - N}\dif x; \mathbb{R}^N\bigr),
\]
thus, up to a subsequence, $u_n \to u$ almost everywhere in $\mathbb{R}^N$. From the continuity of the map defined by \eqref{eqSuperposition}, we have that, as \(n \to \infty\)
\[
 \vert u_n\vert ^p\to \vert u\vert ^p \text{ strongly in } L^{\frac{2}{p}}\bigl(\vert x\vert ^{\frac{N + \alpha}{p} - N}\dif x; \mathbb{R}^N\bigr),
\]
\[
\text{ and } \vert u_n\vert ^{p-2}u_n\to \vert u\vert ^{p-2}u \text{ strongly in } L^{\frac{2}{p - 1}} (\vert x\vert ^{\frac{N + \alpha}{p} - N}\dif x; \mathbb{R}^N).
\]
By the Stein--Weiss inequality \cite{SW}, we deduce that, as \(n \to \infty\)
\[
I_\alpha*\vert u_n\vert ^p\to I_\alpha*\vert u\vert ^p\text{ strongly in } L^{\frac{2}{2 - p}}(\vert x\vert ^{-\frac{N + \alpha - p N}{2 - p}}\dif x; \mathbb{R}^N),
\]
thus
\begin{multline*}
  \bigl(I_\alpha*\vert u_n\vert ^p\bigr)\vert u_n\vert ^{p-2}u_n\to \bigl(I_\alpha \ast \vert u\vert ^p\bigr)\vert u\vert ^{p-2}u \\
  \text{ strongly  in }
  L^2 \bigl(\vert x\vert ^{N - \frac{N + \alpha}{p}}\dif x; \mathbb{R}^N\bigr) = \Bigl(L^2\bigl(\vert x\vert ^{\frac{N + \alpha}{p} - N}\dif x; \mathbb{R}^N\bigr)\Bigr)'.
\end{multline*}
By the continuous embedding results again, we have
\begin{equation*}
\bigl(I_\alpha*\vert u_n\vert ^p\bigr)\vert u_n\vert ^{p-2}u_n\to \bigl(I_\alpha \ast \vert u\vert ^p\bigr)\vert u\vert ^{p-2}u  \text{ strongly in } \big(H^{1}_{V}(\mathbb{R}^N)\bigr)'.\qedhere
\end{equation*}
\end{proof}
The compact embedding theorems imply
straightforwardly that the functional $J_p$ is well-defined and satisfies the Palais--Smale condition.

\begin{lemma}
Let \(N \ge 1\), \(\alpha \in (0, N)\) and $p>1$. If $\frac{1}{p}>\frac{N-2}{N+\alpha}$ and if
\[
\liminf\limits_{\vert x\vert \to+\infty}\frac{V(x)}{1+\vert x\vert ^{\frac{N+\alpha}{p}-N}}=+\infty,
\]
then the functional
$J_p$ is of class \(C^1\) on \(H^1_V (\mathbb{R}^N)\) and satisfies the Palais--Smale condition, that is,
any sequence \((u_n)_{n \in \mathbb{N}}\) in \(H^{1}_{V}(\mathbb{R}^N)\) such that $(J_p(u_n))_{n \in \mathbb{N}}$ is bounded, and $J_{p}'(u_n)\to 0$ strongly in  $(H^{1}_{V}(\mathbb{R}^N))'$ as \(n \to \infty\) has a subsequence that converges strongly in \(H^{1}_{V}(\mathbb{R}^N)\).
\end{lemma}
\begin{proof}

To prove that the functional $J_p$ is of continuously differentiable, we only need to consider the nonlocal term $\mathcal{G}_p$ of $J_p$, that is,
\[
\mathcal{G}_p(u)=\int_{\mathbb{R}^N}\bigl(I_\alpha \ast \vert u\vert ^p\bigr)\vert u\vert ^p=\int_{\mathbb{R}^N}\big\vert I_{\alpha/2}*\vert u\vert ^p\big\vert ^2.
\]
By Proposition~\ref{lem2.2}, the functional $\mathcal{G}_p$ is continuous on $H^{1}_{V}(\mathbb{R}^N)$ and then the functional $J_p$ is also continuous.  For the continuous differentiability, we observe that by Proposition~\ref{lem2.2} again
the map $\mathcal{G}_p$ is G\^{a}teaux-differentiable on $H^{1}_{V}(\mathbb{R}^N)$ and hence it is continuously Fr\'{e}chet differentiable on that space \cite{W}*{Proposition 1.3} and the first part of the conclusion follows.

Suppose now that $(u_n)_{n \in \mathbb{N}}$ is a Palais--Smale sequence for the functional $J_p$, that is, as \(n \to \infty\)
\[
(J_p(u_n))_{n \in \mathbb{N}} \text{ is bounded}\qquad \text{ and }\qquad J_{p}'(u_n)\to 0 \text{ strongly in } \big(H^{1}_{V}(\mathbb{R}^N)\bigr)'.
\]
First, we observe that the sequence $(u_n)_{n\in\mathbb{N}}$ is bounded in the space $H_{V}^{1}(\mathbb{R}^N)$, because
\[
\Bigl(\frac{1}{2}-\frac{1}{2p}\Bigr)\Vert u_n\Vert _{V}^2=J_p(u_n)-\frac{1}{2p}\langle J_{p}'(u_n),u_n\rangle=J_p(u_n)+o(\Vert u_n\Vert _V).
\]
Up to a subsequence, we can assume that the sequence \((u_n)_{n \in \mathbb{N}}\) converges weakly to some function \(u\in H_{V}^{1}(\mathbb{R}^N)\).
By Proposition~\ref{lem2.2}, we have \(\mathcal{G}_p'(u_n) \to \mathcal{G}_p'(u)\) as \(n \to \infty\) strongly in \((H^1_V (\mathbb{R}^N))'\) --- that is, the map $\mathcal{G}_p'$ is weak to strong type. It follows then that, as \(n \to \infty\),
\begin{equation*}
\Vert u_{n}-u\Vert _{V}^{2} = \langle J_{p}'(u_{n})-J_{p}'(u), u_{n}-u\rangle +\frac{1}{2p}\langle \mathcal{G}_p'(u_n)-\mathcal{G}_p'(u),(u_{n}-u)\rangle
\to 0,
\end{equation*}
which concludes the proof.
\end{proof}

\section{Ground states and multiplicity solutions}
\label{sectionGroundstates}
We first give a proof of Theorem~\ref{thm1.2} by minimization of the Sobolev quotient and then prove the multiplicity
result Theorem~\ref{thm1.3} by the fountain theorem at the end of this section.

\begin{proof}[Proof of Theorem~\ref{thm1.2}]
We are going to find a  minimizer $u \in H^1_V (\mathbb{R}^N)$ for the infimum $\theta_p$,
defined by
\begin{equation*}
\theta_p:= \inf\; \Bigl\{\int_{\mathbb{R}^N}\vert \nabla u\vert ^2+V\vert u \vert^2 \st  u \in H^1 V (\mathbb{R}^N) \text{ and } \int_{\mathbb{R}^N}\bigl(I_\alpha \ast \vert u\vert ^p\bigr)\vert u\vert ^p =1\Bigr\}\,;
\end{equation*}
once this will be done a nontrivial solution $v$ of equation \eqref{eq1.1} will be obtained after a rescaling, more precisely, by taking \(v=\theta_p^{1/(2p-2)}u\).

Let $(u_n)_{n\in\mathbb{N}}$ in \(H^1_V (\mathbb{R}^N)\) be a minimizing sequence for $\theta_p$, then
\[
\Vert u_n\Vert ^{2}_V\to \theta_p\quad \text{ and } \int_{\mathbb{R}^N}\bigl(I_\alpha*\vert u_n\vert ^p\bigr)\vert u_n\vert ^p \dif x=1.
\]
Since the sequence $(u_n)_{n\in\mathbb{N}}$ is bounded in \(H^1_V (\mathbb{R}^N)\), we can assume  without loss of generality that $u_n\rightharpoonup u$ weakly in \(H^1_V (\mathbb{R}^N)\) as \(n \to \infty\).
By the weakly lower semi-continuity of the norm, we know that
\begin{equation}\label{eq3.3}
\Vert u\Vert _V\leq\liminf\limits_{n\to\infty}\Vert u_n\Vert _V.
\end{equation}
On the other hand, by Proposition~\ref{lem2.2}, we deduce that
\[
I_{\alpha/2}*\vert u_n\vert ^p\to I_{\alpha/2}*\vert u\vert ^p \text{ strongly in } L^{2}(\mathbb{R}^N),
\]
and thus, as $n\to\infty$,
\[
\mathcal{G}_p(u_n)=\int_{\mathbb{R}^N}\big\vert I_{\alpha/2}*\vert u_n\vert ^p\big\vert ^2\to \int_{\mathbb{R}^N}\big\vert I_{\alpha/2}*\vert u\vert ^p\big\vert ^2=\mathcal{G}_p(u).
\]
Therefore,
\[
\int_{\mathbb{R}^N}(I_{\alpha}*\vert u\vert ^p)\vert u\vert ^p \dif x=1,
\]
which leads to $u\neq 0$ and $\Vert u\Vert _{V}^{2}\geq \theta_p$ by the definition of $\theta_p$. This, together with the inequality \eqref{eq3.3},
implies that $\Vert u\Vert ^{2}_{V}=\theta_p$. Therefore, $u$ is a minimizer for $\theta_p$.
\end{proof}
\begin{remark}
In fact, the nontrivial solution obtained above is a positive solution with least energy, that is, a groundstate, see \cites{W}.
\end{remark}

In the remainder of this section, we prove  Theorem~\ref{thm1.3} on the  multiplicity
results by the fountain theorem.

\begin{proof}[Proof of Theorem~\ref{thm1.3}]
We are going to apply Bartsch's fountain theorem \cite{B} as it is presented in \cite{W}*{Theorem 3.6}.
We consider the action of the group \(\mathbb{Z}/2\mathbb{Z} = \{-1, 1\}\) on the space \(H^1_V (\mathbb{R}^N)\) defined for \(g \in \{-1, 1\}\) and \(u \in H^1_V (\mathbb{R}^N)\) by multiplication.
This action is continuous and isometric: for every \(g \in \{-1, 1\}\) and \(u \in H^1_V (\mathbb{R}^N)\), \(\Vert  g u \Vert _V = \Vert  u \Vert _V\). The functional $J_p$ is invariant under this action of the group \(\{-1, 1\}\) since it is an even functional.
Moreover, by the Borsuk--Ulam theorem \cite{W}*{Theorem D.17}, the action of \(\{-1, 1\}\) on $\mathbb{R}$ is admissible for the fountain theorem, that is, every continuous odd map $f : \partial U\to \mathbb{R}^{k-1}$ has a zero, where $k\geq 2$ and $U$ is an open bounded symmetric neighborhood of $0$ in $\mathbb{R}^k$.

We choose an orthonormal basis $(e_j)_{j\geq0}$ of \(H^1_V (\mathbb{R}^N)\) and define $X_j:=\mathbb{R}e_j$. Then the spaces $X_j$ are invariant, that is for every $g\in \{-1, 1\}$, $gX_j=X_j$, and $H^1_V (\mathbb{R}^N)=Y_k\oplus Z_k$, where $Y_k:=\bigoplus_{j=0}^{j=k}X_j$ and $Z_k:=\overline{\bigoplus_{j\geq k}X_j}^{\Vert  \cdot\Vert _V}$.
For every $k\geq 1$,  we set
\[
  B_k:=\bigl\{u\in Y_k\st \Vert u\Vert _V\leq \rho_k\bigr\},
  \quad N_k:=\bigl\{u\in Z_k\st \Vert u\Vert _V=r_k\bigr\}
\]
with $\rho_k>r_k>0$ and we define a sequence of minimax level,
\[
c_k:=\inf\limits_{\varphi\in \Gamma_k}\max\limits_{u\in B_k} J_{p}(\varphi(u)),
\]
where
\begin{equation*}
\Gamma_k:=\big\{\varphi\in C(B_k,H^1_V (\mathbb{R}^N))\st \varphi \text{ is even and }\varphi= \operatorname{id} \text{ on }\partial B_k \big\}.
\end{equation*}
To apply the fountain theorem, we still need  to verify the following conditions $(A_1)$ and $(A_2)$:
\begin{equation}\tag{$A_1$}\label{conditions1}
a_k:=\max\limits_{u\in \partial B_k} J_{p}(u)\leq 0
\end{equation}
\begin{equation}\tag{$A_2$}\label{conditions2}
b_k:=\inf\limits_{u\in N_k} J_{p}(u)\to +\infty\text{ as }k\to \infty.
\end{equation}
Since the unit ball in the finite dimensional linear subspace \(Y_k\) is a compact set, we deduce that the continuous functional \(\mathcal{G}_p\) achieves a positive minimum \(\sigma_k\) on that set.
On the finite-dimensional space $Y_k$, for any $u\in Y_k$ with $\Vert u\Vert _V=\rho_k$,
\[
J_{p}(u)=\frac{1}{2}\Vert u\Vert _{V}^2-\frac{1}{2p}\Vert u\Vert _{V}^{2p}\,\mathcal{G}_p\Bigl(\frac{u}{\Vert u\Vert _V}\Bigr)\leq \frac{1}{2}\rho^{2}_{k}-\frac{\sigma_k}{2p}\rho_{k}^{2p}.
\]
Thus the condition \eqref{conditions1} follows for sufficiently large $\rho_k$ since $p>1$.

We now turn to consider \eqref{conditions2}. We define
\[
\beta_k:=\sup\,\bigl\{\Vert I_{\alpha/2} \ast \abs{u}^p \Vert _{L^2} \st u\in Z_k\text{ and } \Vert u\Vert _V=1\bigr\}.
\]
We show that $\beta_k\to 0$ as $k\to \infty$ with minor modification following \cite{W}*{Proof of Lemma 3.8}.
Observe that $0<\beta_{k+1}\leq\beta_{k}$, so that $\beta_k \to\beta\geq 0$, as $k\to\infty$. By the definition of  $\beta_k$, we know that for every $k\geq 0$, there exists $u_k\in Z_k$ such that
\[
\Vert u_k\Vert _V=1\text{ and  } \bigl\Vert I_{\alpha/2} \ast \abs{u_k}^p\bigr\Vert _{L^{2}}>\frac{\beta_k}{2}.
\]
By definition of $Z_k$, we have $u_k\rightharpoonup 0$ weakly in \(H^1_V (\mathbb{R}^N)\).
Thus by the weak to strong convergence property of Proposition~\ref{lem2.2}, we deduce that
\(I_{\alpha/2} \ast \abs{u_k}^p \to I_{\alpha/2} \ast \abs{u}^p\) as \(k \to \infty\) strongly in \(L^2 (\mathbb{R}^N)\). Therefore $\beta=0$.

For every $u \in Z_k$,
\begin{equation*}
J_{p}(u) \geq \frac{1}{2}\Vert u\Vert _{V}^2-\frac{\beta_{k}^{2}}{2 p}\Vert u\Vert _{V}^{2p}.
\end{equation*}
Set $r_k:=1/(\beta_{k})^{1/(p - 1)}$, then we have
\begin{equation*}
J_p(u)\geq \Bigl(\frac{1}{2}-\frac{1}{2p}\Bigr)\frac{1}{\beta_k^{\frac{2}{p - 1}}} \to +\infty \text{ as } k\to\infty.
\end{equation*}
The condition \eqref{conditions2} holds, it thus follows from the fountain theorem that $(c_k)_{k  \in \mathbb{N}}\) is a sequence of critical values of $J_p$ tending to \(+\infty\).
This concludes the proof.
\end{proof}

\section{Existence of nodal solution with least energy}
\label{sectionNodal}

In this section, we shall prove the existence of nodal solutions by minimization method on the Nehari nodal set defined by
 \[
 \mathcal{M}_p=\bigl\{u\in H^1_V (\mathbb{R}^N) \st  u^+\neq 0 \ne u^-\text{ and } \langle J_{p}'(u),u^+\rangle =\langle J_{p}'(u),u^-\rangle =0\bigr\}.
 \]
It is obvious that all the sign-changing solutions are contained in $\mathcal{M}_p$.
We are going to study whether it is possible to obtain a least energy nodal solution by finding a minimizer for 
\[
  c_p:=\inf\limits_{u\in\mathcal{M}_p} J_p(u).
\]

The following lemma plays an essential role in showing the existence of the minimizer for $c_p$. The proof follows the strategy of
\citelist{\cite{GV}*{Proof of proposition 3.2}\cite{Y}*{Lemma 3.2}}.
\begin{lemma}\label{lem4.1}
Let $p>2$. For any $u \in H^1_V (\mathbb{R}^N)$ with $u^{\pm}\neq0$, there exists a unique pair $(\Bar{t},\Bar{s})\in(0,+\infty)^2$ such that $\Bar{t}u^++\Bar{s}u^-\in\mathcal{M}_p$ and if $u\in\mathcal{M}_p$, then $J_p(u)\geq J_p(tu^++su^-)$ for any $t\geq0,s\geq0$.
\end{lemma}
\begin{proof}
We define the function $\Phi_p:[0,+\infty)^2\to \mathbb{R}$ for each \(s, t \in [0, +\infty)\) by
\begin{equation}
\label{phi}
\Phi_p(t,s):= J_p(t^{\frac{1}{p}}u^++s^{\frac{1}{p}}u^-)
=\frac{t^{\frac{2}{p}}}{2}\Vert u^+\Vert _{V}^{2}+\frac{s^{\frac{2}{p}}}{2}\Vert u^-\Vert _{V}^{2}-\frac{1}{2p}\int_{\mathbb{R}^N}
\vert I_{\alpha/2}*(t\vert u^+\vert ^p+s\vert u^-\vert ^p)\vert ^2,
\end{equation}
where $u=u^++u^-$ with $u^{\pm}\neq0$.
The condition $t^{\frac{1}{p}}u^++s^{\frac{1}{p}}u^-\in\mathcal{M}_p$ is equivalent to $\nabla\Phi_p(t,s)=0$
with $t>0,s>0$. It is sufficient to prove that there exists a unique critical point for the function $\Phi_p$ on the domain $(0,+\infty)^2$.

By the definition of $\Phi_p$,
\begin{equation*}
\Phi_p(t,s) \leq\frac{t^{\frac{2}{p}}}{2}\Vert u^+\Vert _{V}^{2}+\frac{s^{\frac{2}{p}}}{2}\Vert u^-\Vert _{V}^{2}-\frac{t^2}{2p}\int_{\mathbb{R}^N}
\bigabs{I_{\alpha/2}*\vert u^+\vert ^p}^2 -\frac{s^2}{2p}\int_{\mathbb{R}^N}\bigabs{I_{\alpha/2}*\vert u^-\vert ^p}^2,
\end{equation*}
from which we can get that
\begin{multline*}
\lim\limits_{t^2+s^2\to+\infty}\Phi_p(t,s)
\leq \lim\limits_{t^2+s^2\to+\infty}\bigg(\frac{t^{\frac{2}{p}}}{2}\Vert u^+\Vert _{V}^{2}-\frac{t^2}{2p}\int_{\mathbb{R}^N}
\bigabs{I_{\alpha/2}*\vert u^+\vert ^p}^2\\
\quad +\frac{s^{\frac{2}{p}}}{2}\Vert u^-\Vert _{V}^{2}-\frac{s^2}{2p}\int_{\mathbb{R}^N}\bigabs{I_{\alpha/2}*\vert u^-\vert ^p}^2\bigg)=-\infty.
\end{multline*}
Therefore, $\Phi_p$ must have at least one global maximum point on \([0, \infty) \times [0, \infty)\).

Since the quadratic form 
\[
 (t, s) \mapsto \int_{\mathbb{R}^N}
\bigabs{I_{\alpha/2}*(t\vert u^+\vert ^p+s\vert u^-\vert ^p)}^2
\]
is positive definite, the function \(\Phi_p\) is strictly concave. In particular, any critical point is a maximum point and there is at most one maximum point. 

The conclusion follows provided that we can rule out that this 
maximum point is on the boundary of $[0,+\infty)^2$.
Suppose that $(t_0,0)$ with $t_0\ge 0$ is the global maximum point of $\Phi_p$, then $\frac{\partial\Phi_p(t_0,0)}{\partial t}\le 0$. However, a direct computation shows that 
\begin{equation*}
\frac{\partial \Phi_p(t_0,s)}{\partial s} \Big\vert_{s = 0} = +\infty,
\end{equation*}
Similarly, $\Phi_p$ can not achieve its global maximum on $(0,s)$ for any $s \ge 0$.
\end{proof}

\begin{proof}[Proof of Theorem~\ref{thm1.5} when $p>2$] For the case of $p>2$, our proof, in fact, relies on the compact embedding: \(H^1_V (\mathbb{R}^N)\hookrightarrow L^q(\mathbb{R}^N)$ with $\frac{1}{q}\in(\frac{1}{2}-\frac{1}{N},\frac{1}{2}]\) and can be carried out into two steps. First, we show that $c_p>0$ is attained by some minimizer $w\in\mathcal{M}_p$. Then, we prove the minimizer $w$ for $c_p$ is indeed a critical point of $J_p$, thus being a nodal solution of \eqref{eq1.1}.

\medbreak

\noindent\textbf{Step 1} \emph{The energy level $c_p>0$ is achieved by some minimizer \(w \in \mathcal{M}_p\).}

\medbreak

Let $(u_n)_{n \in \mathbb{N}}\) be a minimizing sequence for $c_p$ in \(\mathcal{M}_p$, namely, $\lim\limits_{n\to\infty}J_p(u_n)=c_p$.
We first observe that
\begin{equation}\label{eq4.3}
\Bigl(\frac{1}{2}-\frac{1}{2p}\Bigr) \Vert u_n\Vert ^2=
J_p(u_n)-\frac{1}{2p}\langle J_{p}'(u_n),u_n\rangle
= J_p (u_n)  \to
c_p,
\end{equation}
from which we know that the sequence $(u_n)_{n \in \mathbb{N}}$ is bounded in \(H^1_V (\mathbb{R}^N)\) and so are the sequences
$(u_{n}^{\pm})_{n \in \mathbb{N}}$. Passing to a subsequence, there exist $u^{\pm}\in H^1_V (\mathbb{R}^N)$ such
that
\[
u_{n}^{\pm}\rightharpoonup u^{\pm}\text{ weakly in } H^1_V (\mathbb{R}^N).
\]

By the constraint $\langle J_{p}'(u_n),u_{n}^{\pm} \rangle=0$, and by the Hardy--Littlewood--Sobolev inequality \cite{LL}*{Theorem 4.3}, which can be seen as a special case of the Stein--Weiss inequality \cite{SW}, we deduce that
\begin{equation}\label{eq4.4}
\begin{split}
C_1\Vert u_{n}^{\pm}\Vert ^{2}_{L^{\frac{2Np}{N+\alpha}}} \le \Vert u_{n}^{\pm}\Vert_V ^2&=\int_{\mathbb{R}^N} \bigl(I_\alpha*\vert u_{n}\vert ^p\bigr)\vert u_{n}^{\pm}\vert ^p \dif x\\
 &\leq C_2\Vert u_n\Vert ^{p}_{L^{\frac{2Np}{N+\alpha}}}\Vert u_{n}^{\pm}\Vert ^{p}_{L^{\frac{2Np}{N+\alpha}}}
 \leq C_3\Vert u_{n}^{\pm}\Vert ^{p}_{L^{\frac{2Np}{N+\alpha}}},
\end{split}
 \end{equation}
which yields, since by our constraint again \(u_n^\pm \ne 0\), that
\begin{equation}
\label{eq4.nonvanish}
  \liminf_{n \to \infty} \Vert u_{n}^{\pm}\Vert _{L^{\frac{2Np}{N+\alpha}}} > 0.
\end{equation}
Since the embedding $H^1_V (\mathbb{R}^N)\hookrightarrow L^{\frac{2Np}{N+\alpha}}(\mathbb{R}^N)$ is compact, we have
\[
u_{n}^{\pm}\to u^{\pm}\text{ strongly  in } L^{\frac{2Np}{N+\alpha}}(\mathbb{R}^N),
\]
and then we deduce from \eqref{eq4.nonvanish} that $u^{\pm}\neq 0$.
Next, by the Hardy--Littlewood--Sobolev inequality,  we see that
\begin{equation}
\label{eqNodConvSWSym}
\int_{\mathbb{R}^N} \bigl(I_\alpha*\vert u^{\pm}_n\vert ^p\bigr)\vert u^{\pm}_n\vert ^p\to\int_{\mathbb{R}^N} \bigl(I_\alpha*\vert u^{\pm}\vert ^p\bigr)\vert u^{\pm}\vert ^p,
\end{equation}
and
\begin{equation}
\label{eqNodConvSWMix}
\int_{\mathbb{R}^N} \bigl(I_\alpha*\vert u_n^+\vert ^p\bigr)\vert u_n^-\vert ^p\to\int_{\mathbb{R}^N} \bigl(I_\alpha*\vert u^+\vert ^p\bigr)\vert u^-\vert ^p.
\end{equation}
Hence, by Lemma~\ref{lem4.1}, there exists a unique pair $(t_0,s_0) $ with $t_0,s_0>0$ such that $t_0u^++s_0u^-\in\mathcal{M}_p$. Moreover, we have
\begin{equation*}
\begin{split}
c_p\leq J_p(t_0u^++s_0u^-)
&\leq \liminf\limits_{n\to\infty}J_p(t_0u^{+}_n+s_0u^{-}_n) \\
&\le \limsup\limits_{n\to\infty}J_p(t_0u^{+}_n+s_0u^{-}_n)
\le \lim\limits_{n\to\infty}J_p(u_n)= c_p.
\end{split}
 \end{equation*}
The second inequality above follows from the weakly lower semi-continuity of the norm and from \eqref{eqNodConvSWSym} and \eqref{eqNodConvSWMix}. 
We conclude by setting $w=t_0u^++s_0u^-$.

\medbreak

\noindent \textbf{Step 2} \emph{$J_{p}'(w)=0$.}

\medbreak

To complete this, we follow the idea of perturbing the functional  in one direction \cite{LW}. This argument seems simpler than previous deformation arguments \cites{AS,WZ}.

Suppose that $w$ is not a critical point, then there exists a function $v\in C^{\infty}_{c}(\mathbb{R}^N)$ such that $\langle J_{p}'(w),v\rangle=-2$.
Since $J_p$ is continuously differentiable, there exists $\delta>0$ small enough such that
\begin{equation}\label{eq4.5}
\langle J_{p}'(tu^++su^-+\epsilon v),v\rangle\leq-1, \qquad \text{ if }\vert t-t_0\vert +\vert s-s_0\vert \leq \delta \text{ and } 0\leq\epsilon\leq\delta.
\end{equation}
We choose a continuous function $\eta:D\to [0,1]$, where $D$ being a bounded domain and is defined by
\[
D:=\big\{(t,s)\in \mathbb{R}^2: \vert t-t_0\vert \leq \delta, \vert s-s_0\vert \leq \delta\big\},
\]
such that
\begin{equation*}
\eta(t,s)=\begin{cases}
1 &\text{ if } \vert t-t_0\vert \leq \frac{\delta}{4} \text{ and } \vert s-s_0\vert \leq \frac{\delta}{4},  \\
0 &\text{ if } \vert t-t_0\vert \geq \frac{\delta}{2} \text{ or } \vert s-s_0\vert \geq \frac{\delta}{2}.
\end{cases}
\end{equation*}
We define $Q\in C(D,H^1_V (\mathbb{R}^N))$ for \((t, s) \in D\) by
\[
Q(t,s)=tu^++su^-+\delta\eta(t,s)v.
\]
and $h:D\to \mathbb{R}^2$ for \((t, s) \in D\) as
\[
h(t,s):=\bigl(\langle J_{p}'(Q(t,s)),Q(t,s)^+\rangle,\langle J_{p}'(Q(t,s)),Q(t,s)^-\rangle\bigr).
\]
The map $h$ is continuous because the map $u\mapsto u^+$ is continuous in \(H^1_V (\mathbb{R}^N)\). If $\vert t-t_0\vert =\delta$, or $\vert s-s_0\vert =\delta$,
then $\eta=0$ by its definition, therefore $Q(t,s)=tu^++su^-$, which implies that $h(t,s)\neq(0,0)$ by Lemma~\ref{lem4.1}.
As a consequence, the Brouwer topological degree $\deg(h, \operatorname{int}(D), 0)$ is well defined and
$\deg(h, \operatorname{int}(D), 0)=1$, thus there exists a pair $(t_1,s_1)\in \operatorname{int}(D)$ such that $h(t_1,s_1)=(0,0)$.
Thus $Q(t_1,s_1)\in \mathcal{M}_p$, and then, it follows from the definition of \(c_p\) that 
\begin{equation}\label{eq4.6}
J_{p}(Q(t_1,s_1))\geq c_p.
\end{equation}
On the other hand, from equation \eqref{eq4.5} we arrive at
 \begin{equation}\label{eq4.7}
 \begin{split}
 J_p(Q(t_1,s_1))&=J_p(t_1u^++s_1u^-)+\int_{0}^1 \langle J_{p}'(t_1u^++s_1u^-+\rho \delta\eta(t_1,s_1)v),\delta \eta(t_1,s_1)v\rangle \dif \rho \\
 &\leq J_p(t_1u^++s_1u^-)-\delta\eta(t_1,s_1).
 \raisetag{14pt}
 \end{split}
 \end{equation}
If $(t_1,s_1)\neq(t_0,s_0)$, we know from Lemma~\ref{lem4.1} that $J_p(t_1u^++s_1u^-)<J_p(t_0u^++s_0u^-)=c_p$, thus from inequality \eqref{eq4.7}
\[
J_p(Q(t_1,s_1))\leq J_p(t_1u^++s_1u^-)<c_p.
\]
If $(t_1,s_1)=(t_0,s_0)$, then $\eta(t_1,s_1)=1$, follows from \eqref{eq4.7} we also have
\[
J_p(Q(t_1,s_1))\leq c_p-\delta <c_p,
\]
which contradicts inequality \eqref{eq4.6} in any case.
\end{proof}

We bring to the attention of the reader that the assumptions on the potential \(V\) are only used to ensure the compactness of the embedding \(H^1_V (\mathbb{R}^N)\hookrightarrow L^\frac{2 N p}{N + \alpha}(\mathbb{R}^N)\).

\medbreak

The case  $p=2$ is more complicated since we have neither a property similar to Lemma~\ref{lem4.1} nor an estimate like \eqref{eq4.4} to guarantee $u^\pm\neq0$, where $u=u^++u^-$ is the weak limit of a minimizing sequence.
To find a nodal solution with least energy for the quadratic case, we follow the idea of \cites{GMV} of employing equation \eqref{eq1.1} with $p>2$ as a regularisation for the quadratic equation \eqref{eq1.1} and then pass to the limit as $p\searrow 2$. We start our proof by showing that  Nehari nodal set $\mathcal{M}_2$ is not empty.

\begin{lemma}
\label{lemmaNehariNotEmpty}
One has $\mathcal{M}_2\neq\emptyset$. In particular, $c_2 < +\infty$.
\end{lemma}
\begin{proof}
We are going to construct a function $w\in H^1_V (\mathbb{R}^N)$ with $w^\pm\neq0$ such that the following linear system admits a solution $(t,s)$ with $t,s>0$,
\begin{equation}\label{eqs1}
  \begin{pmatrix}
 \int_{\mathbb{R}^N} (I_\alpha *\vert w^+\vert ^2)\vert w^+\vert ^2  &  \int_{\mathbb{R}^N} (I_\alpha *\vert w^+\vert ^2)\vert w^-\vert ^2 \\
 \int_{\mathbb{R}^N} (I_\alpha *\vert w^+\vert ^2)\vert w^-\vert ^2 &  \int_{\mathbb{R}^N} (I_\alpha* \vert w^-\vert ^2)\vert w^-\vert ^2
\end{pmatrix}
\begin{pmatrix}
t^2\\ s^2
\end{pmatrix}
=\begin{pmatrix}
\Vert w^+\Vert _{V}^2\\\Vert w^-\Vert _{V}^2
\end{pmatrix}.
\end{equation}
The conclusion will then follow since $tw^++sw^-\in\mathcal{M}_p$.
By Cramer's Rule, it is sufficient to find a function $w \in H^1_V (\mathbb{R}^N)$ with $w^\pm\neq0$ such that
\begin{equation}
\label{eq2}
\frac{\int_{\mathbb{R}^N} (I_\alpha *\vert w^+\vert ^2)\vert w^-\vert ^2}{\int_{\mathbb{R}^N} (I_\alpha *\vert w^-\vert ^2)\vert w^-\vert ^2}
 < \frac {\Vert w^+\Vert _{V}^2}{\Vert w^-\Vert _{V}^2}
< \frac{\int_{\mathbb{R}^N} (I_\alpha *\vert w^+\vert ^2)\vert w^+\vert ^2}{\int_{\mathbb{R}^N} (I_\alpha *\vert w^+\vert ^2)\vert w^-\vert ^2}.
\end{equation}

Let $U\in C^{1}(\mathbb{R}^N)\backslash\{0\}$  such that $U\geq 0$ and \(\supp U \subset B(0, 1)\). We choose $a_+, a_- \not\in$ supp $U$ and define
 \[
 w_\sigma(x):=U(\tfrac{x-a_+}{\sigma})-U(\tfrac{x-a_-}{\sigma}).
 \]
Since the function $U$ has compact support, we know that $w_{\sigma}^+(x)=U(\frac{x-a_+}{\sigma})$ and $w_{\sigma}^-(x)=-U(\frac{x-a_-}{\sigma})$ for sufficiently small $\sigma$. To end the proof, we show that the estimate \eqref{eq2} holds as $\sigma$ becomes small enough. In fact,
\begin{equation*}
\begin{split}
 \Vert w_{\sigma}^\pm\Vert _{V}^2&=\int_{\mathbb{R}^N}\sigma^{N-2} \vert \nabla U(x)\vert ^2+\sigma^NV(a_{\pm}+\sigma x) U^2(x) \dif x\\
 &=\sigma^{N-2} \Bigl(\int_{\mathbb{R}^N}\vert \nabla U\vert ^2 + O (\sigma^{2})\Bigr),
\end{split}
 \end{equation*}
\begin{equation*}
\begin{split}
\int_{\mathbb{R}^N} (I_\alpha *\vert w_{\sigma}^\pm\vert ^2)\vert w_{\sigma}^\pm\vert ^2
&= \int_{\mathbb{R}^N}\int_{\mathbb{R}^N}\frac{A_\alpha\vert U(\frac{y-a_{\pm}}{\sigma})\vert ^2\vert U(\frac{x-a_{\pm}}{\sigma})\vert ^2}{\vert x-y\vert ^{N-\alpha}}\dif y \dif x\\
&= \sigma^{N+\alpha}\int_{\mathbb{R}^N} (I_\alpha *\vert U\vert ^2)\vert U\vert ^2,
\end{split}
\end{equation*}
and when \(\sigma \le \abs{a_+ - a_-}/4\),
\begin{equation*}
\begin{split}
\int_{\mathbb{R}^N} (I_\alpha *\vert w_{\sigma}^+\vert ^2)\vert w_{\sigma}^-\vert ^2
&=  \sigma^{2N}\int_{\mathbb{R}^N}\int_{\mathbb{R}^N}\frac{A_\alpha\vert U(y)\vert ^2\vert U(x)\vert ^2}{\vert (a_-+\sigma x)-(a_++\sigma y)\vert ^{N-\alpha}}\dif y \dif x\\
&\leq \sigma^{2N}\int_{\mathbb{R}^N}\int_{\mathbb{R}^N}\frac{2^{N-\alpha}A_\alpha\vert U(y)\vert ^2\vert U(x)\vert ^2}{\vert a_--a_+\vert ^{N-\alpha}}\dif y \dif x.
\end{split}
\end{equation*}
We observe that in \eqref{eq2}, since \(\alpha < N\), the left-hand side goes to \(0\) as \(\sigma \to 0\), the middle term converges to a positive constant and the right-hand side diverges to \(+\infty\).
The inequality \eqref{eq2} holds for sufficiently small $\sigma$  and thus the  system \eqref{eqs1} has a solution $(t,s)$ such that $t,s>0$, that is, $\mathcal{M}_2\neq\emptyset$.
\end{proof}

\begin{proof}[Proof of Theorem~\ref{thm1.5} when $p=2$]
Let $(u_{p_n})_{n\in\mathbb{N}}\subset H^1_V (\mathbb{R}^N)$ be a sequence of least energy nodal solution for the equation \eqref{eq1.1} with \(\frac{1}{p_n} > \frac{N - 2}{N + \alpha}\) and $p_n\searrow 2$ as \(n \to \infty\). In particular, we have
$J_{p_n}(u_{p_n})=c_{p_n}$, and the function $u_{p_n}$ satisfies the equation
\[
-\Delta u_{p_n}+Vu_{p_n}=\big(I_{\alpha}*\vert u_{p_n}\vert ^{p_n}\bigr)\vert u_{p_n}\vert ^{p_n-2}u_{p_n}.
\]
We first show that $\Vert u_{p_n}\Vert _V$ is bounded both from below and above.  In fact, by a direct computation, we see that for every \(n \in \mathbb{N}\), by the Hardy--Littlewood--Sobolev and by the Sobolev inequality,
\begin{equation*}
\Vert u_{p_n}\Vert _{V}^2
=\int_{\mathbb{R}^N} \big(I_\alpha*\vert u_{p_n}\vert ^{p_n}\bigr)\vert u_{p_n}\vert ^{p_n} \leq C\Big(\int_{\mathbb{R}^N}\vert u_{p_n}\vert ^{\frac{2Np_n}{N+\alpha}}\Big)^{\frac{N+\alpha}{N}}\leq C_1\Vert u_{p_n}\Vert _{V}^{2p_n},
\end{equation*}
where the constant $C_1$ can be taken independently of $p_n$ since $(p_n)_{n \in \mathbb{N}}$ remains bounded. It follows that
\begin{equation}
\label{eqNodalBounded0}
\liminf\limits_{n\to\infty}\Vert u_{p_n}\Vert _V>0.
\end{equation}
On the other hand, thanks to Lemma~\ref{lemmaNehariNotEmpty} above, we can take $w\in \mathcal{M}_2$, and define $w_{p_n}=t^{{1/p_n}}w^++s^{{1/p_n}}w^-$, where $(t^{{1/p_n}},s^{{1/p_n}})$ is given by Lemma~\ref{lem4.1}.
Then,
$
 w_{p_n}\in\mathcal{M}_{p_n},$ and $J(w_{p_n})\geq c_{p_n}$.
Since $\Phi_p(t,s)\to-\infty$ as $(t,s)\to+\infty$ uniformly in $p$ in bounded sets and $\Phi_p\to \Phi_2$ as $p\to 2$ uniformly over compact subsets of $[0,+\infty)^2$, we have
$
 t^{1/p_n},s^{1/p_n}\to 1,$ and therefore $J_{p_n}(w_{p_n})\to J_2(w).
$
Since $w$ is an arbitrary function in $\mathcal{M}_2$, we deduce that
 \begin{equation}\label{eq4.9}
 \limsup\limits_{n\to\infty} c_{p_n}\leq c_2 < +\infty,
 \end{equation}
and thus
\[
\Vert u_{p_n}\Vert _{V}^2=\frac{1}{\frac{1}{2}-\frac{1}{2p_n}}\Big(J_{p_n}(u_{p_n})-\frac{1}{2p_n} \langle J_{p_n}'(u_{p_n}),u_{p_n}\rangle\Big)=\frac{2p_nc_{p_n}}{p_n-1}\leq 4c_2 + o (1).
\]
In particular, $(\Vert u_{p_n}\Vert _V)_{n \in \mathbb{N}}$ is bounded from above. It follows that there exists some function $u \in H^1_V (\mathbb{R}^N)$ such that $u_{p_n}\rightharpoonup u$ weakly in $H^1_V (\mathbb{R}^N)$ as $n \to \infty$.
By the compactness of the embedding $H^1_V (\mathbb{R}^N)\hookrightarrow L^{\frac{4N}{N+\alpha}}(\mathbb{R}^N)$ (Proposition~\ref{embedding}), we have
\[
u_{p_n}\to u \text{ strongly in } L^\frac{4N}{N+\alpha}(\mathbb{R}^N) \qquad \text{ and } \qquad u_{p_n}\to u \text{ almost everywhere in }\mathbb{R}^N,
\]
so that $(u_{p_n})_{n\in\mathbb{N}}$ is bounded in $L^\frac{4N}{N+\alpha}(\mathbb{R}^N)$. Moreover, by interpolation through H\"ol\-der's inequality, we have
\begin{equation}\label{eq4.1.2}
\begin{split}
\Vert u_{p_n}\Vert _{L^{\frac{4N}{N+\alpha}}}
&\leq \Vert u_{p_n}\Vert _{L^{\frac{2Np_n}{N+\alpha}}}^{\lambda_n}\Vert u_{p_n}\Vert _{L^2}^{1-\lambda_n} \leq \Vert u_{p_n}\Vert _{L^{\frac{2Np_n}{N+\alpha}}}^{\lambda_n}\big(C\Vert u_{p_n}\Vert _{V}\bigr)^{1-\lambda_n},
\end{split}
\end{equation}
where $\lambda_n\in(0,1)$ satisfies that
\[
\frac{1}{\;\frac{4N}{N+\alpha}\;}=\frac{\lambda_n}{\; \frac{2Np_n}{N+\alpha}\;}+\frac{1-\lambda_n}{2},
\]
that is, $\lambda_n=\frac{N-\alpha}{Np_n-N-\alpha}\frac{p_n}{2}\to 1$ as $n\to\infty$ and the constant $C$ can be chosen independently of $u_{p_n}$.
Similarly, taking $q=\frac{2N}{N-2}$ for $N\geq3$, and $\frac{4N}{N+\alpha}<q<+\infty$ for $N=1,2$, we have
\begin{equation}\label{eq4.1.3}
\Vert u_{p_n}\Vert _{L^{\frac{2Np_n}{N+\alpha}}}
\leq \Vert u_{p_n}\Vert _{L^q}^{1-\mu_n} \Vert u_{p_n}\Vert _{L^{\frac{4N}{N+\alpha}}}^{\mu_n}\\
\leq \bigl(C\Vert  u_{p_n}\Vert _{V}\bigr)^{1-\mu_n}\Vert u_{p_n}\Vert _{L^{\frac{4N}{N+\alpha}}}^{\mu_n}
\end{equation}
with
\[
\mu_n=\frac{\frac{N+\alpha}{2Np_n}-\frac{1}{q}}{\frac{N+\alpha}{4N}-\frac{1}{q}}\to 1, \qquad \text{ as } n\to\infty.
\]
 Taking limit on the both sides of \eqref{eq4.1.2} and \eqref{eq4.1.3} and combining the boundedness of $\Vert u_{p_n}\Vert _V$ from below and above, we obtain that
\begin{equation*}
\lim\limits_{n\to\infty}\int_{\mathbb{R}^N}\vert u_{p_n}\vert ^\frac{2Np_n}{N+\alpha}=
\lim\limits_{n\to\infty}\int_{\mathbb{R}^N}\vert u_{p_n}\vert ^\frac{4N}{N+\alpha}
=\int_{\mathbb{R}^N}\vert u\vert ^\frac{4N}{N+\alpha}.
\end{equation*}
Thus, we get that (see for example \cite{W}*{Proposition 4.2.6})
\begin{equation}\label{eq.4.14}
\vert u_{p_n}\vert ^{p_n} \to \vert u\vert ^2\text { strongly in } L^{\frac{2N}{N+\alpha}}(\mathbb{R}^N),
\end{equation}
which, together with the  Stein--Weiss inequality \cite{SW}, yields that
\begin{equation}\label{eq.4.15}
I_\alpha*\vert u_{p_n}\vert ^{p_n}\to  I_\alpha*\vert u\vert ^2 \text{ strongly in } L^{\frac{2N}{N-\alpha}}(\mathbb{R}^N).
\end{equation}
Similarly to \eqref{eq4.4}, we have
\begin{equation*}
 C_1\Vert u_{p_n}\Vert _{L^{\frac{4N}{N+\alpha}}}^2\leq\Vert u_{p_n}\Vert _{V}^2
 =\int_{\mathbb{R}^N} \bigl(I_\alpha*\vert u_{p_n}\vert ^{p_n}\bigr)\vert u_{p_n}\vert ^{p_n}\leq C \big\Vert \vert u_{p_n}\vert ^{p_n}\big\Vert _{L^{\frac{2N}{N+\alpha}}}^{2},
\end{equation*}
which implies, by taking limit on both sides, that $\Vert u\Vert _{L^\frac{4N}{N+\alpha}}\geq C>0$, that is $u\neq 0$.

For large $n$, we choose $q$ as in \eqref{eq4.1.3}, we employ the interpolation inequalities again, and we get
\begin{equation*}
\Vert u_{p_n}\Vert _{L^{(p_n-1)\frac{4N}{N+\alpha}}}
\leq\Vert u_{p_n}\Vert _{L^q}^{1-\lambda_n}\Vert u_{p_n}\Vert _{L^{\frac{4N}{N+\alpha}}}^{\lambda_n}\\
\leq \big(C\Vert  u_{p_n}\Vert _{V}\bigr)^{1-\lambda_n}\Vert u_{p_n}\Vert _{L^{\frac{4N}{N+\alpha}}}^{\lambda_n},
\end{equation*}
where
\[
\lambda_n=\frac{\frac{N+\alpha}{4N(p_n-1)}-\frac{1}{q}}{\frac{N+\alpha}{4N}-\frac{1}{q}}\to 1,
\] that is
$ \vert u_{p_n}\vert ^{p_n-2}u_{p_n}$ is bounded in $L^{\frac{4N}{N+\alpha}}(\mathbb{R}^N)$, it converges to $u$ weakly
in the space $L^{\frac{4N}{N+\alpha}}(\mathbb{R}^N)$ \cite{W2}*{Proposition 5.4.7}.
Therefore,
\begin{equation*}
\begin{split}
\langle J_{2}'(u),\psi\rangle
&=\int_{\mathbb{R}^N} \nabla u\cdot\nabla \psi+Vu\psi- (I_{\alpha}*\vert u\vert ^2)u\psi\\
&= \lim\limits_{n\to\infty} \int_{\mathbb{R}^N} \nabla u_{p_n}\cdot\nabla \psi+Vu_{p_n}\psi- (I_{\alpha}*\vert u_{p_n}\vert ^{p_n})\vert u_{p_n}\vert ^{p_n-2}u_{p_n}\psi\\
&= \langle J_{p_n}'(u_{p_n}),\psi\rangle=0.
\end{split}
\end{equation*}
which means that $u$ is a weak solution of the quadratic Choquard equation \eqref{eq1.1} since the function $\psi\in H^{1}_V(\mathbb{R}^N)$ is arbitrary.
Moreover, we deduce from the convergences \eqref{eq.4.14} and \eqref{eq.4.15} that
\begin{equation}
\label{eqNodalConvergenceNorms}
  \Vert u_{p_n}\Vert _{V}^2=\int_{\mathbb{R}^N}(I_\alpha*\vert u_{p_n}\vert ^{p_n})\vert u_{p_n}\vert ^{p_n}\to \int_{\mathbb{R}^N}(I_\alpha*\vert u\vert ^{2})\vert u\vert ^{2}=\Vert u\Vert _{V}^2,
\end{equation}
and thus $u_{p_n}\to u$ strongly in \(H^1_V (\mathbb{R}^N)\) as \(n \to \infty\).

We are now in a position to finish our proof by showing that $u^\pm \neq 0$ and $J_2(u)$ is the least among all the nodal solutions of the quadratic Choquard equation.
By \eqref{eqNodalBounded0} and \eqref{eqNodalConvergenceNorms}, we have \(u \ne 0\).
Without loss of generality, let us assume by contradiction that $ u^{+}\neq 0$ and that \(u^- = 0\). Set for each \(n \in \mathbb{N}\)
\[
v_{p_n}:=\frac{u_{p_n}^{-}}{\Vert u_{p_n}^{-}\Vert _{V}^{{2/p_n}}},
\]
then from the equality $\langle J_{p_n}'(u_{p_n}),u^{-}_{p_n}\rangle=0$, we get that for each \(n \in \mathbb{N}\)
\begin{equation*}
\int_{\mathbb{R}^N} \bigl(I_\alpha*\vert u_{p_n}\vert ^{p_n}\bigr)\vert v_{p_n}\vert ^{p_n}=1.
\end{equation*}
Since we have assumed that $u^{-}=0$, we have $u_{p_n}^{-}\to 0$ strongly in \(H^1_V (\mathbb{R}^N)\). By Young's inequality, we know that
\[
\Vert v_{p_n}\Vert _V=\Vert u_{p_n}^-\Vert _{V}^{1-\frac{2}{p_n}}\leq\Bigl(1-\frac{2}{p_n}\Bigr)\Vert u^{-}_{p_n}\Vert _V+\frac{2}{p_n},
\]
which yields that the sequence $(v_{p_n})_{n \in \mathbb{N}}$ is bounded in \(H^1_V (\mathbb{R}^N)\). The compactness of the embedding $H^1_V (\mathbb{R}^N)\hookrightarrow L^{\frac{4N}{N+\alpha}}(\mathbb{R}^N)$ (Proposition~\ref{embedding}) implies in turn that $(v_{p_n})_{n \in \mathbb{N}}$ converges to some $v$ strongly in
$L^{\frac{4N}{N+\alpha}}(\mathbb{R}^N)$, which, together with  \eqref{eq4.1.3} by replacing $u_{p_n}$ with $v_{p_n}$ and the boundness of $\Vert v_{p_n}\Vert _V$, implies that
\[
\limsup_{n \to \infty} \int_{\mathbb{R}^N}\bigl\vert \vert v_{p_n}\vert ^{p_n}\bigr\vert ^{\frac{2N}{N+\alpha}} < +\infty,
\]
thus, it follows from  \cite{W2}*{Proposition 5.4.7} again that
\[
\vert v_{p_n}\vert ^{p_n}\rightharpoonup \vert v\vert ^2 \text{ in } L^{\frac{2N}{N+\alpha}}(\mathbb{R}^N).
\]
Combining the strong convergence of \eqref{eq.4.15}, we deduce that
 \begin{equation}
 \label{eqNodalRieszuv}
 \int_{\mathbb{R}^N} \bigl(I_\alpha*\vert u\vert ^2\bigr)\vert v\vert ^{2}=\lim\limits_{n\to\infty}\int_{\mathbb{R}^N} \bigl(I_\alpha*\vert u_{p_n}\vert ^{p_n}\bigr)\vert v_{p_n}\vert ^{p_n}=1.
 \end{equation}
On the other hand, by the definition of $v_{p_n}$ and by the strong convergence of \((u_{p_n}^+)_{n \in \mathbb{N}}\) to \(u\) in \(H^1 (\mathbb{R}^N)\),
we have $uv=0$ almost everywhere on \(\mathbb{R}^N\). Since $u$ is a nontrivial nonnegative weak solution to the Choquard equation, it is a classical solution (following \cite{MorozVanSchaftingen2013JFA}*{Theorem 3}) and thus, by the classical strong maximum principle for second order elliptic operators, $u > 0$ everywhere on $\mathbb{R}^N$ and thus $v=0$, which is
a contradiction with \eqref{eqNodalRieszuv}.

In particular, we have \(u \in \mathcal{M}_2\), and thus \(J_2 (u) \ge c_2\).
On the other hand, by \eqref{eq4.9} and the strong convergence of \((u_{p_n})_{n \in \mathbb{N}}\)  we have
\begin{equation*}
\begin{split}
J_2(u)
&=\frac{1}{4}\int_{\mathbb{R}^N}\vert \nabla u\vert ^2+V\vert u \vert^2 \\
&= \lim\limits_{n\to\infty} \Bigl(\frac{1}{2}-\frac{1}{2p_n}\Bigr)\int_{\mathbb{R}^N}\vert \nabla u_{p_n}\vert ^2+V\vert u_{p_n}\vert^2
= \limsup_{n \to \infty}  J_{p_n}(u_{p_n})=\limsup_{n \to \infty} c_{p_n}\leq c_2;
\end{split}
\end{equation*}
this concludes the proof.
\end{proof}
\begin{remark}\label{rem4.3}
In fact, in the case of $p\geq 2$, we have $c_p>c_{0,p}$ where $c_{0,p}$ is the energy level of the groundstates, since any groundstate solution should have constant sign.
However, the question of whether or not or when the estimate $c_p>2c_{0,p}$ holds is open;
in the case of constant potential this estimate was crucial for the compactness.
\end{remark}

Finally, we prove that the energy level \(c_p\) is degenerate when \(p < 2\).

\begin{proposition}\label{prop4}
For $p<2$, we have $c_p=c_{0,p}$. Then the energy functional $J_p$ does not achieve its minimum on
the Nehari nodal set.
\end{proposition}
\begin{proof}
We observe that if $u\in\mathcal{N}_p$, then $\vert u\vert \in \mathcal{N}_p$, where $\mathcal{N}_p$ denote the Nehari manifold, that is,
\[
\mathcal{N}_p:=\bigl\{u\in H^{1}_V(\mathbb{R}^N): u\neq 0 \text{ and } \langle J'_{p}(u),u\rangle=0\bigr\}.
\]
With this notation, we know that $c_{0,p}=\inf_{u\in\mathcal{N}_p} J_p(u)$ (see Remark~\ref{rem4.3} and \cite{W}). Since $\mathcal{M}_p\subset\mathcal{N}_p$, thus we get that  $c_p\geq c_{0,p}$.
 In fact, we shall show the reverse inequality holds. By a density argument, it follows that
\[
c_{0,p}=\inf\{J_p(u): u\in \mathcal{N}_p\cap C^{1}_{0}(\mathbb{R}^N) \text{ and } u\geq 0 \text{ on }\mathbb{R}^N\}.
\]
Let $u\in \mathcal{N}_p\cap C^{1}_{c}(\mathbb{R}^N)$ and $u\geq0$ on $\mathbb{R}^N$. We choose a point $a\not\in$ supp $u$ and a function
$\psi\in C^{1}_{c}(\mathbb{R}^N)\backslash\{0\}$ such that $\psi\geq 0$ and we define as in \cites{GV} for each $\sigma>0$ the function $u_\sigma: \mathbb{R}^N\to \mathbb{R}$ by
\[
u_\sigma(x)=u(x)-\sigma^{\frac{2}{2-p}}\psi(\tfrac{x-a}{\sigma}).
\]
Then, $u_{\sigma}^{+}=u$ for sufficiently small $\sigma$. By a direct computation, $tu_{\sigma}^{+}+su_{\sigma}^{-}\in\mathcal{M}_p$ if and only if
\begin{equation}\label{eq4.12}
\left\{
 \begin{aligned}
(t^{2-p}-t^p)\int_{\mathbb{R}^N}\bigl(I_\alpha \ast \vert u\vert ^p\bigr)\vert u\vert ^p &=s^p\sigma^{N+\frac{2p}{2-p}}K_\sigma,\\
s^{2-p}\int_{\mathbb{R}^N}\vert \nabla \psi(y)\vert ^2+\sigma^2V(a+\sigma y)\psi^2(y)\dif y&=t^pK_\sigma+s^p\sigma^{\alpha+\frac{2p}{2-p}}\int_{\mathbb{R}^N}\bigl(I_\alpha *\vert \psi\vert ^p\bigr)\vert \psi\vert ^p
\end{aligned}\right.
\end{equation}
where $K_\sigma=\int_{\mathbb{R}^N}\bigl(I_\alpha \ast \vert u\vert ^p\bigr)(a+\sigma y)\vert \psi(y)\vert ^pdy$.

Observe that thet system \eqref{eq4.12} has a unique solution when $\sigma=0$. By the implicit function theorem, for $\sigma>0$ small enough there exists a pair $(t_\sigma,s_\sigma)\in(0,+\infty)^2$ that solves the system \eqref{eq4.12} and
\[
\lim\limits_{\sigma\to 0}t_\sigma=1,\qquad \lim\limits_{\sigma\to 0}s_\sigma=\bigg(\frac{\bigl(I_\alpha \ast \vert u\vert ^p\bigr)(a)\int_{\mathbb{R}^N}\vert \psi\vert ^p }{\int_{\mathbb{R}^N}\vert \nabla \psi\vert ^2}\bigg)^{\frac{1}{2-p}}.
\]
Since $\frac{4}{2-p}+N-2>0$, we have $u_{\sigma}^-\to 0$ in \(H^1_V (\mathbb{R}^N)\) as $\sigma\to 0$, and thus $t_\sigma u_{\sigma}^{+}+s_{\sigma}u_{\sigma}^{-}\to u$  in \(H^1_V (\mathbb{R}^N)\). Hence, we have that
\[
c_p\leq J_p(t_\sigma u_{\sigma}^{+}+s_{\sigma}u_{\sigma}^{-})\to J_p(u), \qquad \text{ as } \sigma\to 0,
\]
which implies that $c_p\leq c_{0,p}$ since $u\in \mathcal{N}_p$ is arbitrary.

We assume that $u\in \mathcal{M}_p$ minimizes the functional $J_p$ on the Nehari nodal set $\mathcal{M}_p$. Since $c_{0,p}=c_p$, thus $u$ also minimizes $J_p$ on the  Nehari manifold $\mathcal{N}_p$. By regularity theory for the Choquard equation and by the strong maximum principle, either $u>0$ or $u<0$, which contradicts with $u\in\mathcal{M}_p$.
\end{proof}

\section{Poho\v{z}aev identity}
\label{sectionPohozaev}

This section is devoted to the proof of a Poho\v{z}aev identity for the Choquard equation \eqref{eq1.1}.

\begin{theorem}[Poho\v{z}aev identity]\label{thm5.1}
Let $N\geq 3$, $V\in C^1(\mathbb{R}^N,[0,+\infty))$. If the function $u\in W^{2,2}_{\mathrm{loc}}(\mathbb{R}^N)\cap H^{1}_{V}(\mathbb{R}^N)$ is a solution to the Choquard equation \eqref{eq1.1} such that
\[
\int_{\mathbb{R}^N} \vert x\cdot \nabla V(x)\vert \,\vert u (x) \vert^2 \dif x+\int_{\mathbb{R}^N}\bigl(I_\alpha \ast \vert u\vert ^p\bigr)\vert u\vert ^p <+\infty,
\]
then
\begin{equation}\label{eqph}
\frac{N-2}{2}\int_{\mathbb{R}^N} \vert \nabla u\vert ^2 +\frac{1}{2}\int_{\mathbb{R}^N} \big(NV(x)+x\cdot\nabla V(x)\bigr) \vert u (x) \vert^2\dif x
=\frac{N+\alpha}{2p}\int \bigl(I_\alpha \ast \vert u\vert ^p\bigr)\vert u\vert ^p .
\end{equation}
\end{theorem}

Here, $\eta \cdot \zeta$ denotes the canonical scalar product of vectors $\eta, \zeta \in \mathbb{R}^N$.

\begin{proof}
[Proof of Theorem~\ref{thm5.1}]
We take $\varphi\in C^{1}_{c}(\mathbb{R}^N)$ such that $\varphi=1$ on $B(0,1)$.
Since the function $\varphi$ has compact support, we can define a function $v_\lambda\in H_{V}^{1}(\mathbb{R}^N)$ for $\lambda\in(0,+\infty)$ by
\[
v_\lambda(x):=\varphi(\lambda x)\,x\cdot\nabla u(x).
\]
By testing the Choquard equation \eqref{eq1.1} against the function
$v_\lambda$, we have
\[
 \int_{\mathbb{R}^N}\nabla u\cdot \nabla v_\lambda+V uv_\lambda =\int_{\mathbb{R}^N}\bigl(I_{\alpha}*\vert u\vert ^p\bigr)\vert u\vert ^{p-2}uv_\lambda.
\]
We compute the square term for $\lambda>0$. By the definition of $v_\lambda$, the chain rule and by the Gauss integral formula, we get that
\begin{equation}
\begin{split}
\int_{\mathbb{R}^N}Vuv_\lambda
&=\int_{\mathbb{R}^N} V(x)u(x)\varphi(\lambda x)\,x\cdot\nabla u(x)\dif x\\
&=\int_{\mathbb{R}^N}V(x)\varphi(\lambda x)\,x\cdot\nabla\big(\tfrac{1}{2}\vert u\vert ^2\bigr)(x) \dif x\\
&=-\int_{\mathbb{R}^N} \bigl(NV(x)\varphi(\lambda x)+ V(x)\;(\lambda x)\cdot\nabla\varphi(\lambda x)+ x\cdot\nabla V(x)\varphi(\lambda x)\bigr) \frac{\vert u(x)\vert ^2}{2}\dif x.
\end{split}
\end{equation}
In view of the various boundedness assumptions, Lebesgue's dominated convergence theorem applies and gives us
\[
\lim\limits_{\lambda\to 0}\int_{\mathbb{R}^N} V(x)uv_\lambda =-\frac{1}{2}\int_{\mathbb{R}^N} \bigl(NV(x)+x\cdot\nabla V(x)\bigr) \vert u (x) \vert^2\dif x
\]
In view of the assumption $u\in W^{2,2}_{\mathrm{loc}}(\mathbb{R}^N)$, we can perform an integration by parts
 \begin{equation*}
 \begin{split}
 \int_{\mathbb{R}^N}\nabla u\cdot \nabla v_\lambda
 &=\int_{\mathbb{R}^N}\varphi(\lambda x)\Bigl(\vert \nabla u(x)\vert ^2+x\cdot\nabla\big(\tfrac{\vert \nabla u\vert ^2}{2}\bigr)(x)\Bigr) \dif x\\
 &\qquad\qquad\qquad +\int_{\mathbb{R}^N}\big(\lambda \nabla u(x)\cdot \nabla\varphi(\lambda x)\bigr)\big(x\cdot\nabla u(x)\bigr) \dif x\\
 &=-\int_{\mathbb{R}^N}\bigl((N-2)\varphi(\lambda x)+\lambda x\cdot\nabla\varphi(\lambda x)\bigr)\frac{\vert \nabla u(x)\vert ^2}{2}\dif x\\
 &\qquad\qquad\qquad +\int_{\mathbb{R}^N}\big( \nabla u(x)\cdot \nabla\varphi(\lambda x)\bigr)\big(\lambda x\cdot\nabla u(x)\bigr) \dif x
 \end{split}
 \end{equation*}
Since $\vert (\eta\cdot\zeta)(\eta\cdot\xi)\vert \leq \vert \eta\vert ^2\vert \zeta\vert \vert \xi\vert $ for any $\zeta,\xi,\eta\in\mathbb{R}^N$, we have for each \(x \in \mathbb{R}^N\)
 \[
 \vert \big(\nabla u(x)\cdot \nabla\varphi(\lambda x)\bigr)\big(\lambda  x\cdot\nabla u(x)\bigr)\vert \leq \vert \nabla u(x)\vert ^2\vert \lambda x\vert  \vert \nabla \varphi(\lambda x)\vert \leq \vert \nabla u(x)\vert ^2 \sup_{z \in \mathbb{R}^N} \vert z \vert \, \vert \nabla \varphi (z)\vert.
 \]
 By Lebesgue's dominated convergence theorem again, we have, since $u\in H^1_{V}(\mathbb{R}^N)$,
 \[
 \lim\limits_{\lambda\to 0}\int_{\mathbb{R}^N}\nabla u\cdot \nabla v_\lambda =-\frac{N-2}{2}\int_{\mathbb{R}^N}\vert \nabla u\vert ^2.
 \]
Finally, by symmetry and integration by parts
\begin{equation*}
\begin{split}
\int_{\mathbb{R}^N}&\bigl(I_\alpha \ast \vert u\vert ^p\bigr)\vert u\vert ^{p-2}uv_\lambda \\
&=\int_{\mathbb{R}^N}\int_{\mathbb{R}^N} I_\alpha(x-y)\vert u(y)\vert ^p\varphi(\lambda x)\,x\cdot\nabla\big(\tfrac{\vert u\vert ^p}{p}\bigr)(x) \dif x \dif y\\
&=\frac{1}{2}\int_{\mathbb{R}^N}\int_{\mathbb{R}^N} I_\alpha(x-y)\Big(\vert u(y)\vert ^p \varphi(\lambda x)\,x\cdot\nabla\big(\tfrac{\vert u\vert ^p}{p}\bigr)(x)\\
&\qquad \qquad \qquad \qquad\qquad \qquad \qquad \qquad +\vert u(x)\vert ^p\varphi(\lambda y)y\cdot\nabla\big(\tfrac{\vert u\vert ^p}{p}\bigr)(y)\Big) \dif x \dif y\\
&=-\int_{\mathbb{R}^N}\int_{\mathbb{R}^N}I_\alpha(x-y)\vert u(y)\vert ^p\big( N\varphi(\lambda x)+\lambda x\cdot\nabla\varphi(\lambda x)\bigr)\frac{\vert u(x)\vert ^p}{p}\dif x \dif y\\
&\qquad+\frac{N-\alpha}{2p}\int_{\mathbb{R}^N}\int_{\mathbb{R}^N}I_\alpha(x-y)\vert u(y)\vert ^p\frac{(x-y)\cdot(x\varphi(\lambda x)-y\varphi(\lambda y))}{\vert x-y\vert ^2}\vert u(x)\vert ^p \dif x \dif y.
\end{split}
\end{equation*}
For any $\lambda>0$ and \(x, y \in \mathbb{R}^N\),
\begin{equation}
\begin{split}
\Big\vert \frac{(x-y)\cdot(x\varphi(\lambda x)-y\varphi(\lambda y))}{\vert x-y\vert ^2}\Big\vert 
&=\Big\vert \frac{ (\lambda x-\lambda y)\cdot(\lambda x\varphi(\lambda x)-\lambda y\varphi(\lambda y))}{\vert \lambda x-\lambda y\vert ^2}\Big\vert \\
&\le \sup_{z, w \in \mathbb{R^N}} \Big\vert \frac{ (w-z)\cdot(w\varphi(w)-z\varphi(z))}{\vert w - z\vert ^2}\Big\vert 
< +\infty.
\end{split}
\end{equation}
We can thus apply Lebesgue's dominated convergence theorem to conclude that
\[
\lim\limits_{\lambda\to 0}\int_{\mathbb{R}^N}\bigl(I_\alpha \ast \vert u\vert ^p\bigr)\vert u\vert ^{p-2}uv_\lambda=-\frac{N+\alpha}{2p}\int_{\mathbb{R}^N}\bigl(I_\alpha \ast \vert u\vert ^p\bigr)\vert u\vert ^{p}.
\]
Hence, the identity \eqref{eqph} holds.
\end{proof}

\begin{remark} The Poho\v{z}aev identity implies some nonexistence results for the Choquard equation \eqref{eq1.1}. In general, if
\begin{gather}
  \text{either}\qquad \big(2V(x)+x\cdot\nabla V(x)\bigr)\big(N-2-\tfrac{N+\alpha}{p}\bigr)\geq0,\\
  \text{or}\qquad \big((N-\tfrac{N+\alpha}{p})V(x)+x\cdot\nabla V(x)\bigr)\big(N-2-\tfrac{N+\alpha}{p}\bigr)\geq0, 
\end{gather}
then the Choquard equation \eqref{eq1.1} has no nontrivial solutions satisfying the regularity and boundedness assumptions of Theorem~\ref{thm5.1}. In particular, if $V(x)=\vert x\vert ^\beta$ is homogeneous, then the Choquard equation \eqref{eq1.1} has no such solution if
$p\in(1,\max\{1,\frac{N+\alpha}{N+\beta}\}]\cup[\frac{N+\alpha}{N-2},+\infty)$.
\end{remark}

\section*{Acknowledgement}

Jean Van Schaftingen was supported by the Projet de Recherche (Fonds de la Recherche Scientifique--FNRS) T.1110.14 ``Existence and asymptotic behavior of solutions to systems of semilinear elliptic partial differential equations''.
Jiankang Xia acknowledges the support of the China Scholarship Council and the hospitality the Universit\'e catholique de Louvain (Institut de Recherche en Math\'ematique et en Physique).

\end{document}